\documentclass[11pt]{article}
\usepackage[a4paper,margin=1in]{geometry}
\usepackage{amsmath,amssymb,amsthm}
\usepackage{natbib}
\usepackage{tikz}
\usepackage{graphicx}

\newtheorem{theorem}{Theorem}
\newtheorem{example}[theorem]{Example}
\newtheorem{proposition}[theorem]{Proposition}
\newtheorem{corollary}[theorem]{Corollary}
\newtheorem{definition}[theorem]{Definition}
\newtheorem{remark}{Remark}

\title{Pulse Graphs: Prime-Activated Boolean Dynamics on Directed Graphs}
\author{
Pakin Methawisal\\
Mahidol University International College, Mahidol University\\
\texttt{pakin.met@student.mahidol.edu}
}
\date{}

\begin{document}
\maketitle

\begin{abstract}
We study synchronous Boolean dynamics on finite loopless directed graphs in which a vertex is active at the next time step exactly when its number of active in-neighbors is prime. We call these systems \emph{Pulse Graphs}. Let \(L(n)\) denote the largest attractor period realizable on \(n\) vertices. Exhaustive enumeration gives
\[
L(1),\ldots,L(5)=1,1,1,3,9.
\]
Our main result determines the exponential order of the maximum period:
\[
2^{n-3}-1\leq L(n)\leq2^n-n-1
\qquad(n\geq5).
\]

The lower bound is obtained by implementing a maximal-length affine feedback register using prime-count logic gates. For \(n\geq6\), the construction is loopless, has maximum in-degree five, and uses only \(O(n)\) edges. For odd \(n\geq7\), a period-3 control module improves the lower bound to
\[
3(2^{n-4}-1).
\]

For complete directed graphs, we derive an exact update formula, classify all attractors as fixed points or complement two-cycles, prove that every orbit reaches its eventual attractor within three updates, and count the attractors explicitly. We also derive the activation probability under independent random inputs. For sparse random directed graphs, the associated prime-Poisson mean-field map undergoes a nondegenerate fold at
\[
c_\ast\approx3.824963,
\qquad
\rho_\ast\approx0.368241,
\]
with local bistability immediately above the threshold.
\end{abstract}

\section{Introduction}

Boolean networks are finite-state dynamical systems in which interacting components take values in \(\{0,1\}\) and update according to prescribed local rules. They have been studied as abstract models of regulatory, computational, and network dynamics since the early work of Kauffman \citep{kauffman1969metabolic}. Their behavior depends on the directed graph, the local update functions, and the update schedule.

In this paper, we choose the prime numbers as the activating input counts. This gives a simple arithmetic rule that is nonmonotone: for example, a vertex receiving three active inputs becomes active, whereas one receiving four active inputs becomes inactive. More generally, increasing the number of active inputs can switch the output between active and inactive. Prime activation therefore provides a natural setting for studying how an elementary arithmetic condition can generate nontrivial global dynamics on a directed graph.

The rule is totalistic because the next state of a vertex depends only on the number of its active in-neighbors, rather than on their identities. Totalistic rules on general graphs have been studied as extensions of cellular-automaton dynamics \citep{marr2009outer,goles2021totalistic}.

Pulse Graphs differ from traditional monotone threshold systems. For symmetric threshold networks updated synchronously, periodic attractors have period at most two \citep{goles1980periodic}. Kuhlman et al.\ obtained a similar restriction for synchronous bi-threshold systems, while showing that asynchronous bi-threshold systems may have longer periodic orbits \citep{kuhlman2011bifurcations}. The prime-activation rule is nonmonotone and does not satisfy the assumptions of those results. This leads to the extremal question that motivates the present paper: how large can an attractor period be on \(n\) vertices?

This extremal question is related to work on finite dynamical systems with prescribed interaction graphs, including rank and periodic rank \citep{gadouleau2016simple,gadouleau2018rank,gadouleau2020influence}. Interaction graphs of automata networks representing the same dynamics up to isomorphism are studied by \citet{bridoux2023interaction}, whereas the bounded-degree realization problem is studied by \citet{aracena2026bounded}. The register construction below is also related to maximal-length feedback registers and, more broadly, to structured long traversals of finite state spaces, including Gray-code constructions \citep{zierler1959linear,fredricksen1982full,mutze2023gray}. A key distinguishing restriction in the present setting is that every vertex uses the same prime-count local rule.

The central question of this paper is how much dynamical complexity can be generated by this single arithmetic update rule, and how strongly that complexity depends on the underlying graph. We approach this question from two complementary directions. On general directed graphs, prime activation can implement Boolean operations and feedback mechanisms that produce very long periodic orbits. On highly symmetric graphs like complete graphs, however, the same rule can collapse to much simpler dynamics. This contrast shows that the behavior of Pulse Graphs is governed not only by the nonmonotonicity of prime activation, but also by the underlying substrate structure through which that rule is coupled.

\begin{definition}
A \emph{Pulse Graph} is a finite directed graph \(G=(V,E)\) with no self-loops, together with Boolean states
\[
x_t:V\to\{0,1\}, \qquad t=0,1,2,\dots.
\]
For \(v\in V\), let
\[
N^-(v)=\{u\in V:(u,v)\in E\},
\qquad
k_t(v)=\sum_{u\in N^-(v)}x_t(u).
\]
Thus \(k_t(v)\) is the number of active in-neighbors of \(v\) at time \(t\).

\medskip

The update rule is
\[
x_{t+1}(v)=
\begin{cases}
1, & \text{if } k_t(v) \text{ is prime},\\
0, & \text{otherwise}.
\end{cases}
\]

A state \(x\) is a \emph{fixed point} if applying the Pulse update leaves every vertex unchanged. An \emph{attractor of period \(p\)} is a periodic sequence of \(p\) distinct states. Thus a fixed point is an attractor of period \(1\), while an attractor of period at least \(2\) is called a \emph{nontrivial attractor}.

\end{definition}

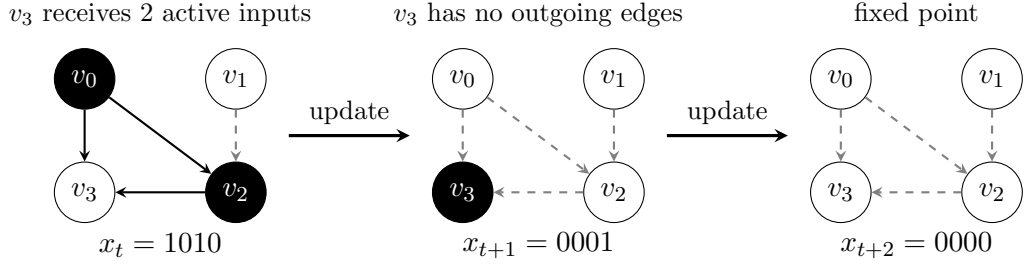
\begin{figure}[ht]
\centering
\begin{tikzpicture}[
    >=stealth,
    active/.style={
        circle,
        draw,
        fill=black,
        text=white,
        minimum size=8mm
    },
    inactive/.style={
        circle,
        draw,
        fill=white,
        text=black,
        minimum size=8mm
    },
    counted/.style={->,thick},
    uncounted/.style={->,thick,dashed,draw=gray}
]

\begin{scope}
    \node[active]   (a0) at (0,1.5) {$v_0$};
    \node[inactive] (a1) at (2,1.5) {$v_1$};
    \node[inactive] (a3) at (0,0) {$v_3$};
    \node[active]   (a2) at (2,0) {$v_2$};

    \draw[counted]   (a0) -- (a3);
    \draw[counted]   (a0) -- (a2);
    \draw[uncounted] (a1) -- (a2);
    \draw[counted]   (a2) -- (a3);

    \node[draw=none] at (1,-0.7) {$x_t=1010$};
    \node[draw=none,align=center] at (1,2.35)
    {\small \(v_3\) receives \(2\) active inputs};
\end{scope}

\draw[->,very thick] (2.7,0.75) -- (4.3,0.75)
node[midway,above,draw=none] {\small update};

\begin{scope}[xshift=5cm]
    \node[inactive] (b0) at (0,1.5) {$v_0$};
    \node[inactive] (b1) at (2,1.5) {$v_1$};
    \node[active]   (b3) at (0,0) {$v_3$};
    \node[inactive] (b2) at (2,0) {$v_2$};

    \draw[uncounted] (b0) -- (b3);
    \draw[uncounted] (b0) -- (b2);
    \draw[uncounted] (b1) -- (b2);
    \draw[uncounted] (b2) -- (b3);

    \node[draw=none] at (1,-0.7) {$x_{t+1}=0001$};
    \node[draw=none,align=center] at (1,2.35)
    {\small \(v_3\) has no outgoing edges};
\end{scope}

\draw[->,very thick] (7.7,0.75) -- (9.3,0.75)
node[midway,above,draw=none] {\small update};

\begin{scope}[xshift=10cm]
    \node[inactive] (c0) at (0,1.5) {$v_0$};
    \node[inactive] (c1) at (2,1.5) {$v_1$};
    \node[inactive] (c3) at (0,0) {$v_3$};
    \node[inactive] (c2) at (2,0) {$v_2$};

    \draw[uncounted] (c0) -- (c3);
    \draw[uncounted] (c0) -- (c2);
    \draw[uncounted] (c1) -- (c2);
    \draw[uncounted] (c2) -- (c3);

    \node[draw=none] at (1,-0.7) {$x_{t+2}=0000$};
    \node[draw=none,align=center] at (1,2.35)
    {\small fixed point};
\end{scope}

\end{tikzpicture}
\caption{Two consecutive Pulse updates on a \(4\)-vertex graph. Solid arrows originate at active vertices. Starting from \(x_t=1010\), only \(v_3\) becomes active. Since \(v_3\) has no outgoing edges, the next state is \(0000\), which is a fixed point of period \(1\).}
\label{fig:basic-pulse-update}
\end{figure}

\begin{proposition}[Basic fixed-point facts]
Let \(G=(V,E)\) be a Pulse Graph.
\begin{enumerate}
    \item If \(N^-(v)=\varnothing\), then \(x_{t+1}(v)=0\) for all \(t\).
    \item The all-ones state is a fixed point if and only if every vertex has prime in-degree.
    \item For every \(n\ge3\), there exists a Pulse Graph on \(n\) vertices whose all-ones state is fixed.
    \item Without self-loops, \(n=3\) is the smallest number of vertices for which the all-ones state can be fixed.
\end{enumerate}
\end{proposition}

\begin{proof}
If \(N^-(v)=\varnothing\), then \(k_t(v)=0\), and \(0\) is not prime.

For the all-ones state, each vertex \(v\) sees exactly \(|N^-(v)|\) active in-neighbors. Hence the all-ones state is fixed exactly when every \(|N^-(v)|\) is prime.

For existence when \(n\ge3\), take the bidirected cycle on \(n\) vertices:
\[
v_i\leftrightarrow v_{i+1}
\qquad
(\text{indices mod }n).
\]
Each vertex has exactly two incoming neighbors, and \(2\) is prime.

Finally, if \(n=1\), every vertex has in-degree \(0\). If \(n=2\), every vertex has in-degree at most \(1\). Neither \(0\) nor \(1\) is prime. Hence \(n=3\) is minimal.
\end{proof}

\begin{remark}
If self-loops are allowed, then \(2\) vertices suffice: let each vertex point to itself and to the other vertex.
\end{remark}

\section{Small Pulse Graphs}

We begin by examining Pulse Graphs on at most five vertices. Because the state space is finite and the update rule is deterministic, every orbit eventually becomes periodic. Exhaustive enumeration then determines the largest possible attractor period for each \(n\leq5\).

\subsection{Eventual Periodicity}

\begin{proposition}
For every finite Pulse Graph, every initial state eventually enters an attractor.
\end{proposition}

\begin{proof}
Let \(n=|V|\). Since each vertex is either active or inactive, the system has exactly \(2^n\) possible states.

Starting from any initial state, consider the orbit
\[
x_0,x_1,x_2,\dots.
\]
Among the first \(2^n+1\) states, at least two must be equal. Because the Pulse update is deterministic, once a state repeats, all subsequent states repeat in the same order. Therefore every orbit eventually enters a periodic attractor.
\end{proof}

\subsection{Exhaustive Enumeration}

\begin{definition}
Let \(L(n)\) denote the maximum attractor period among all labeled directed loopless Pulse Graphs on \(n\) vertices.
\end{definition}

\begin{proposition}[Exact results for \(n\leq5\)]
The maximum attractor periods for Pulse Graphs on at most five vertices are
\[
\begin{array}{c|r|c}
n
& \text{number of labeled loopless directed graphs}
& L(n)\\
\hline
1 & 1                         & 1\\
2 & 4                         & 1\\
3 & 64                        & 1\\
4 & 4{,}096                   & 3\\
5 & 1{,}048{,}576             & 9
\end{array}
\]
\end{proposition}

\begin{proof}
For each \(n\leq5\), all \(2^{n(n-1)}\) labeled loopless directed graphs were enumerated. For every graph, each of its \(2^n\) initial states was iterated until repetition, and the largest resulting attractor period was recorded. Taking the maximum over all graphs gives the values in the table.
\end{proof}

Thus \(n=4\) is the first size at which a nontrivial attractor appears.

\subsection{Examples for Four and Five Vertices}

\subsubsection{A Period-\(3\) Attractor on Four Vertices}

The graph consists of three bidirected edges incident to \(v_2\) and the directed cycle
\[
v_1\to v_0\to v_3\to v_1.
\]
Its edge set is
\[
\begin{aligned}
E=\{& v_2\leftrightarrow v_0,\ v_2\leftrightarrow v_1,\ v_2\leftrightarrow v_3,\\
& v_1\to v_0,\ v_0\to v_3,\ v_3\to v_1 \}.
\end{aligned}
\]

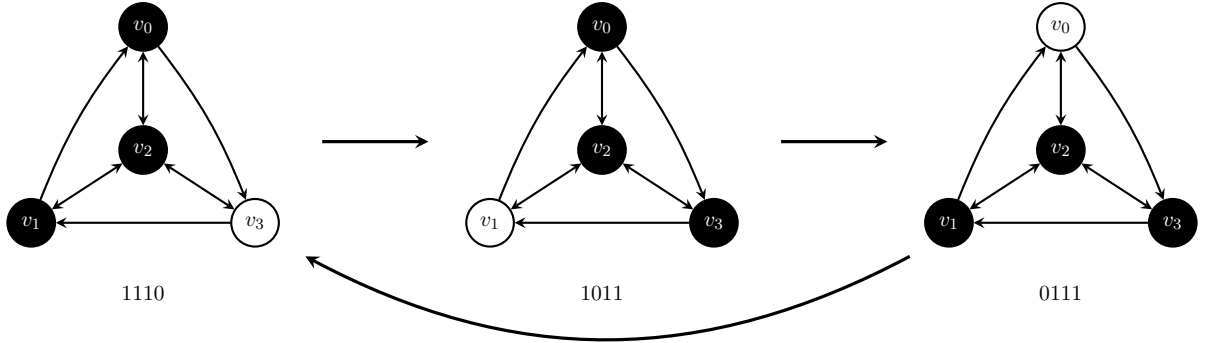
\begin{figure}[ht]
\centering
\begin{tikzpicture}[
>=stealth,
scale=0.74,
transform shape,
active/.style={circle,draw,fill=black,text=white,minimum size=8.5mm},
inactive/.style={circle,draw,fill=white,text=black,minimum size=8.5mm},
every path/.style={thick}
]

\begin{scope}[xshift=0cm]
\node[active]   (a2) at (0,0) {$v_2$};
\node[active]   (a0) at (0,2.2) {$v_0$};
\node[active]   (a1) at (-2,-1.3) {$v_1$};
\node[inactive] (a3) at (2,-1.3) {$v_3$};

\draw[<->] (a2) -- (a0);
\draw[<->] (a2) -- (a1);
\draw[<->] (a2) -- (a3);
\draw[->]  (a1) to[bend left=8] (a0);
\draw[->]  (a0) to[bend left=8] (a3);
\draw[->]  (a3) -- (a1);

\node at (0,-2.55) {$1110$};
\end{scope}

\draw[->,very thick] (3.2,0.15) -- (5.1,0.15);

\begin{scope}[xshift=8.2cm]
\node[active]   (b2) at (0,0) {$v_2$};
\node[active]   (b0) at (0,2.2) {$v_0$};
\node[inactive] (b1) at (-2,-1.3) {$v_1$};
\node[active]   (b3) at (2,-1.3) {$v_3$};

\draw[<->] (b2) -- (b0);
\draw[<->] (b2) -- (b1);
\draw[<->] (b2) -- (b3);
\draw[->]  (b1) to[bend left=8] (b0);
\draw[->]  (b0) to[bend left=8] (b3);
\draw[->]  (b3) -- (b1);

\node at (0,-2.55) {$1011$};
\end{scope}

\draw[->,very thick] (11.4,0.15) -- (13.3,0.15);

\begin{scope}[xshift=16.4cm]
\node[active]   (c2) at (0,0) {$v_2$};
\node[inactive] (c0) at (0,2.2) {$v_0$};
\node[active]   (c1) at (-2,-1.3) {$v_1$};
\node[active]   (c3) at (2,-1.3) {$v_3$};

\draw[<->] (c2) -- (c0);
\draw[<->] (c2) -- (c1);
\draw[<->] (c2) -- (c3);
\draw[->]  (c1) to[bend left=8] (c0);
\draw[->]  (c0) to[bend left=8] (c3);
\draw[->]  (c3) -- (c1);

\node at (0,-2.55) {$0111$};
\end{scope}

\draw[->,very thick] (13.7,-1.9) to[bend left=28] (2.9,-1.9);

\end{tikzpicture}
\caption{The full period-\(3\) attractor on the \(4\)-vertex Pulse Graph. A double-headed edge represents the two corresponding directed edges.}
\label{fig:period-three-graph}
\end{figure}

Under the Pulse update rule,
\[
1110\longmapsto1011\longmapsto0111\longmapsto1110.
\]

Throughout this orbit, \(v_2\) remains active. Among the three outer vertices, exactly one is inactive, and the inactive position moves around the directed cycle
\[
v_3\longrightarrow v_1\longrightarrow v_0\longrightarrow v_3.
\]
The three displayed states are distinct, so this is an attractor of exact period \(3\).

\subsubsection{A Period-\(9\) Attractor on Five Vertices}

The graph contains a bidirected \(K_{2,2}\) between \(\{v_1,v_2\}\) and \(\{v_3,v_4\}\). Its edge set is
\[
\begin{aligned}
E=\{& v_1\leftrightarrow v_3,\ v_1\leftrightarrow v_4,\ v_2\leftrightarrow v_3,\ v_2\leftrightarrow v_4,\\
& v_1\to v_0,\ v_2\to v_0,\ v_1\to v_2,\ v_0\to v_3,\ v_0\to v_4,\ v_3\to v_4 \}.
\end{aligned}
\]

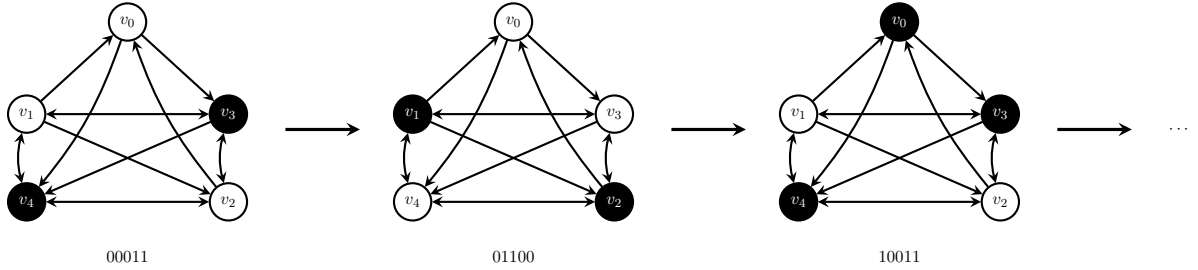
\begin{figure}[ht]
\centering
\begin{tikzpicture}[
>=stealth,
scale=0.58,
transform shape,
active/.style={circle,draw,fill=black,text=white,minimum size=8.5mm},
inactive/.style={circle,draw,fill=white,text=black,minimum size=8.5mm},
every path/.style={thick}
]

\begin{scope}[xshift=0cm]
\node[inactive] (a0) at (0,3) {$v_0$};
\node[inactive] (a1) at (-2.3,0.9) {$v_1$};
\node[inactive] (a2) at (2.3,-1.1) {$v_2$};
\node[active]   (a3) at (2.3,0.9) {$v_3$};
\node[active]   (a4) at (-2.3,-1.1) {$v_4$};

\draw[<->] (a1) -- (a3);
\draw[<->] (a2) -- (a4);
\draw[<->] (a1) to[bend right=14] (a4);
\draw[<->] (a2) to[bend left=14] (a3);
\draw[->]  (a1) -- (a0);
\draw[->]  (a2) to[bend left=10] (a0);
\draw[->]  (a0) -- (a3);
\draw[->]  (a0) to[bend left=10] (a4);
\draw[->]  (a1) -- (a2);
\draw[->]  (a3) -- (a4);

\node at (0,-2.35) {$00011$};
\end{scope}

\draw[->,very thick] (3.6,0.55) -- (5.3,0.55);

\begin{scope}[xshift=8.8cm]
\node[inactive] (b0) at (0,3) {$v_0$};
\node[active]   (b1) at (-2.3,0.9) {$v_1$};
\node[active]   (b2) at (2.3,-1.1) {$v_2$};
\node[inactive] (b3) at (2.3,0.9) {$v_3$};
\node[inactive] (b4) at (-2.3,-1.1) {$v_4$};

\draw[<->] (b1) -- (b3);
\draw[<->] (b2) -- (b4);
\draw[<->] (b1) to[bend right=14] (b4);
\draw[<->] (b2) to[bend left=14] (b3);
\draw[->]  (b1) -- (b0);
\draw[->]  (b2) to[bend left=10] (b0);
\draw[->]  (b0) -- (b3);
\draw[->]  (b0) to[bend left=10] (b4);
\draw[->]  (b1) -- (b2);
\draw[->]  (b3) -- (b4);

\node at (0,-2.35) {$01100$};
\end{scope}

\draw[->,very thick] (12.4,0.55) -- (14.1,0.55);

\begin{scope}[xshift=17.6cm]
\node[active]   (c0) at (0,3) {$v_0$};
\node[inactive] (c1) at (-2.3,0.9) {$v_1$};
\node[inactive] (c2) at (2.3,-1.1) {$v_2$};
\node[active]   (c3) at (2.3,0.9) {$v_3$};
\node[active]   (c4) at (-2.3,-1.1) {$v_4$};

\draw[<->] (c1) -- (c3);
\draw[<->] (c2) -- (c4);
\draw[<->] (c1) to[bend right=14] (c4);
\draw[<->] (c2) to[bend left=14] (c3);
\draw[->]  (c1) -- (c0);
\draw[->]  (c2) to[bend left=10] (c0);
\draw[->]  (c0) -- (c3);
\draw[->]  (c0) to[bend left=10] (c4);
\draw[->]  (c1) -- (c2);
\draw[->]  (c3) -- (c4);

\node at (0,-2.35) {$10011$};
\end{scope}

\draw[->,very thick] (21.2,0.55) -- (22.9,0.55);
\node at (24.0,0.55) {$\cdots$};

\end{tikzpicture}
\caption{The first three states of the period-\(9\) attractor on the \(5\)-vertex Pulse Graph. A double-headed edge represents the two corresponding directed edges. The orbit continues through six more distinct states before returning to \(00011\).}
\label{fig:period-nine-graph}
\end{figure}

The corresponding period-\(9\) orbit is
\[
\begin{aligned}
00011
&\longmapsto
01100
\longmapsto
10011
\longmapsto
01101
\longmapsto
10111\\
&\longmapsto
01111
\longmapsto
11111
\longmapsto
11110
\longmapsto
10110
\longmapsto
00011.
\end{aligned}
\]

All nine states are distinct, so the orbit has exact period \(9\).

The indegrees of \(v_0,\dots,v_4\) are
\[
2,2,3,3,4,
\]
respectively. In particular, from the all-ones state, the first four vertices remain active because they receive either \(2\) or \(3\) active inputs, while \(v_4\) turns off because it receives \(4\) active inputs:
\[
11111\longmapsto11110.
\]
This transition shows the nonmonotonicity of prime activation: increasing the number of active inputs from \(3\) to \(4\) changes the output from active to inactive.

In ordinary threshold systems, additional active inputs generally do not suppress activation \citep{goles2021totalistic}. Pulse Graphs do not have this property.

\section{Long Attractor Periods}
\label{sec:long-periods}

\begin{theorem}[Rotating-register lower bound]
\label{thm:rotating-register}
For every \(n\ge5\),
\[
L(n)\ge n-3.
\]
Consequently, \(L(n)\) is unbounded.
\end{theorem}

\begin{proof}
Set \(m=n-3\). Take three control vertices \(c_0,c_1,c_2\), joined pairwise in both directions and initialized active. Arrange the remaining vertices in the directed cycle
\[
s_0\to s_1\to\cdots\to s_{m-1}\to s_0.
\]
Add \(c_0\to s_i\) for every \(i\). Each signal vertex therefore copies its predecessor: it receives one
active input when its predecessor is inactive and two when its predecessor is active. Thus a single active signal moves once around the cycle and returns after exactly \(m=n-3\) updates, as shown in
Figure~\ref{fig:period-construction}.
\end{proof}

\begin{figure}[htbp]
\centering
\resizebox{\linewidth}{!}{
\begin{tikzpicture}[
    >=stealth,
    active/.style={
        circle,draw,fill=black,text=white,minimum size=8mm
    },
    inactive/.style={
        circle,draw,fill=white,text=black,minimum size=8mm
    },
    counted/.style={->,thick},
    uncounted/.style={->,thick,dashed,draw=gray},
    bidirected/.style={<->,thick}
]

\begin{scope}[name prefix=first-,xshift=0cm]
    \node[active] (c0) at (0,0.4) {$c_0$};
    \node[active] (c1) at (-1,1.6) {$c_1$};
    \node[active] (c2) at (1,1.6) {$c_2$};

    \draw[bidirected] (c0) -- (c1);
    \draw[bidirected] (c0) -- (c2);
    \draw[bidirected] (c1) -- (c2);

    \node[active]   (s0) at (-3,-1.8) {$s_0$};
    \node[inactive] (s1) at (-1,-2.8) {$s_1$};
    \node            (dots) at (1,-2.8) {$\cdots$};
    \node[inactive] (sm) at (3,-1.8) {$s_{m-1}$};

    \draw[counted]   (s0) -- (s1);
    \draw[uncounted] (s1) -- (dots);
    \draw[uncounted] (dots) -- (sm);
    \draw[uncounted] (sm) -- (s0);

    \draw[counted] (c0) -- (s0);
    \draw[counted] (c0) -- (s1);
    \draw[counted] (c0) -- (sm);
\end{scope}

\draw[->,very thick] (3,-0.6) -- (5,-0.6)
    node[midway,above] {update};

\begin{scope}[name prefix=second-,xshift=8cm]
    \node[active] (c0) at (0,0.4) {$c_0$};
    \node[active] (c1) at (-1,1.6) {$c_1$};
    \node[active] (c2) at (1,1.6) {$c_2$};

    \draw[bidirected] (c0) -- (c1);
    \draw[bidirected] (c0) -- (c2);
    \draw[bidirected] (c1) -- (c2);

    \node[inactive] (s0) at (-3,-1.8) {$s_0$};
    \node[active]   (s1) at (-1,-2.8) {$s_1$};
    \node            (dots) at (1,-2.8) {$\cdots$};
    \node[inactive] (sm) at (3,-1.8) {$s_{m-1}$};

    \draw[uncounted] (s0) -- (s1);
    \draw[counted]   (s1) -- (dots);
    \draw[uncounted] (dots) -- (sm);
    \draw[uncounted] (sm) -- (s0);

    \draw[counted] (c0) -- (s0);
    \draw[counted] (c0) -- (s1);
    \draw[counted] (c0) -- (sm);
\end{scope}

\node at (12,-0.6) {$\boldsymbol{\cdots}$};

\begin{scope}[name prefix=last-,xshift=16cm]
    \node[active] (c0) at (0,0.4) {$c_0$};
    \node[active] (c1) at (-1,1.6) {$c_1$};
    \node[active] (c2) at (1,1.6) {$c_2$};

    \draw[bidirected] (c0) -- (c1);
    \draw[bidirected] (c0) -- (c2);
    \draw[bidirected] (c1) -- (c2);

    \node[inactive] (s0) at (-3,-1.8) {$s_0$};
    \node[inactive] (s1) at (-1,-2.8) {$s_1$};
    \node            (dots) at (1,-2.8) {$\cdots$};
    \node[active]   (sm) at (3,-1.8) {$s_{m-1}$};

    \draw[uncounted] (s0) -- (s1);
    \draw[uncounted] (s1) -- (dots);
    \draw[uncounted] (dots) -- (sm);
    \draw[counted]   (sm) -- (s0);

    \draw[counted] (c0) -- (s0);
    \draw[counted] (c0) -- (s1);
    \draw[counted] (c0) -- (sm);
\end{scope}

\draw[->,very thick]
    (13.5,2.5) to[bend right=12]
    node[midway,above] {update}
    (2.5,2.5);
\end{tikzpicture}
}
\caption{Selected updates of the period-\(m\) construction.}
\label{fig:period-construction}
\end{figure}
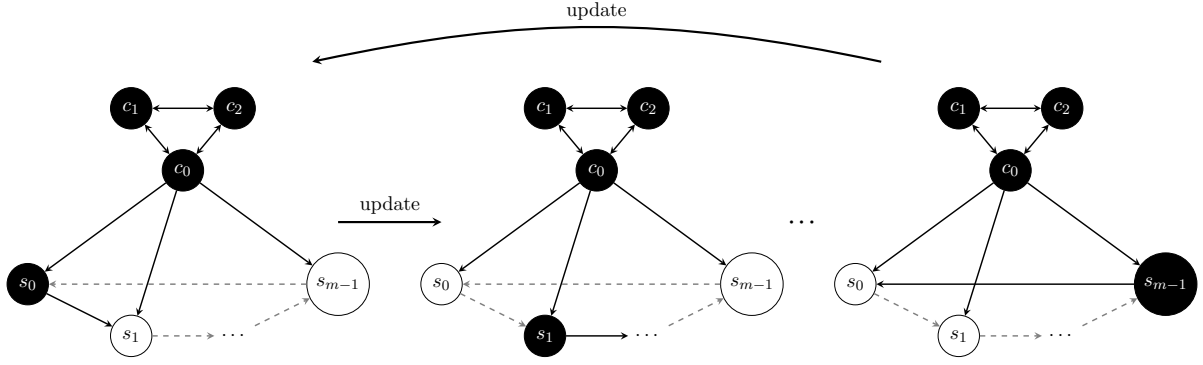

\newpage
\subsection{Prime-Count Logic Gates}
\label{subsec:logic-gates}

The rotating-register construction uses one permanently active input to copy a signal. More generally, fixed active inputs allow the prime-count rule to realize several standard Boolean functions directly.

\begin{proposition}[Prime-count logic gates]
\label{prop:logic-gates}

Suppose that a vertex receives two variable inputs \(x,y\in\{0,1\}\) and
\(q\) inputs from vertices that remain permanently active. Let
\[
F_q(x,y)=
\begin{cases}
1, & q+x+y \text{ is prime},\\
0, & q+x+y \text{ is not prime}.
\end{cases}
\]
For \(0\leq q\leq 4\), the resulting functions are
\[
\begin{array}{c|ccc|c}
q
& F_q(0,0)
& F_q(0,1)=F_q(1,0)
& F_q(1,1)
& \text{Boolean function}\\
\hline
0 & 0 & 0 & 1 & x\wedge y\\
1 & 0 & 1 & 1 & x\vee y\\
2 & 1 & 1 & 0 & \operatorname{NAND}(x,y)\\
3 & 1 & 0 & 1 & \operatorname{XNOR}(x,y)\\
4 & 0 & 1 & 0 & x\oplus y
\end{array}
\]
With one variable input \(x\), one permanently active input computes
\(x\), while three permanently active inputs compute \(\neg x\).
\end{proposition}

\begin{proof}
For two variable inputs, the displayed table is obtained directly by testing primality of \(q+x+y\) for the three possible values \(x+y=0,1,2\).

For one variable input, the counts \(1+x\) are \(1\) and \(2\), so the output equals \(x\). The counts \(3+x\) are \(3\) and \(4\), so the output equals \(\neg x\).
\end{proof}

\begin{remark}
A bidirected triangle initialized in the all-ones state supplies three permanently active control vertices: each control receives exactly two active inputs from the other controls. In particular, Proposition \ref{prop:logic-gates} shows that Pulse Graphs can implement NAND gates and therefore arbitrary feedforward Boolean circuits when inputs are held fixed.
\end{remark}

The exponential-period construction below requires only the copy and XNOR identities.

\subsection{Exponential Attractor Periods}
\label{subsec:exponential-periods}

Maximal-length linear feedback registers constructed from primitive polynomials are classical \citep{zierler1959linear}; full-length nonlinear registers and Gray codes provide related state-space traversals \citep{fredricksen1982full,mutze2023gray}. In the automata-network setting, \citet{aracena2026bounded} realize a fixed point together with a cycle through all remaining configurations using bounded degree. The issue here is different: the local functions cannot be chosen freely, since every vertex must obey the same prime-count activation rule.

The following exponential-period theorem is the main result. The rotating register proves that \(L(n)\) is unbounded; by adding affine feedback, the same three control vertices instead produce a cycle through all but one of the \(2^{n-3}\) signal states.

\begin{theorem}[Exponential attractor periods]
\label{thm:exponential-periods}
For every \(n\geq 5\),
\[
2^{n-3}-1\leq L(n)\leq 2^n-n-1.
\]
Consequently,
\[
L(n)=\Theta(2^n).
\]

For \(n\geq 6\), the lower bound can be realized by a loopless Pulse Graph with maximum in-degree at most five and at most \(5n-15\) directed edges.
\end{theorem}

\begin{proof}
Every state of Hamming weight at most one maps to the all-zero state, since no vertex can receive a prime number of active inputs. The all-zero state is fixed. Hence a nontrivial attractor contains neither the all-zero state nor any of the \(n\) singleton states, and therefore has period at most
\[
2^n-(n+1)=2^n-n-1.
\]
For the lower bound, the case \(n=5\) follows from the exact value
\[
L(5)=9>2^{5-3}-1.
\]

Now suppose \(n\geq 6\), and put
\[
m=n-3.
\]
Work over the field \(\mathbb{F}_2\). By Cohen's prescribed-trace theorem \citep{cohen1990trace}, there exists, for every \(m\geq3\), a primitive element
\[
\alpha\in\mathbb F_{2^m}
\]
with
\[
\operatorname{Tr}_{\mathbb F_{2^m}/\mathbb F_2}(\alpha)=0.
\]
Let
\[
p(z)=z^m+a_{m-1}z^{m-1}+\cdots+a_1z+1
\]
be the minimal polynomial of \(\alpha\) over \(\mathbb F_2\).  Since \(\alpha\) is primitive, \(p\) is a primitive polynomial of degree \(m\). Moreover, in characteristic two, the coefficient \(a_{m-1}\) equals the
trace of \(\alpha\), and hence
\[
a_{m-1}=0.
\]
This zero coefficient will prevent a self-loop at the final register stage.

Let \(A:\mathbb F_2^m\to\mathbb F_2^m\) be the companion transformation defined by
\[
(Ax)_0=x_{m-1}
\]
and
\[
(Ax)_i=x_{i-1}+a_i x_{m-1}, \qquad 1\leq i<m.
\]
The characteristic polynomial of \(A\) is \(p\).  Because \(p\) is
primitive, \(A\) has multiplicative order
\[
2^m-1.
\]
Equivalently, after identifying \(\mathbb F_2^m\) with
\(\mathbb F_{2^m}\), the transformation \(A\) is multiplication by a primitive element. It fixes \(0\) and permutes all \(2^m-1\) nonzero
vectors in one cycle.

\begin{example}[A tap pattern for \(m=3\)]
For \(m=3\), the primitive polynomial \(p(z)=z^3+z+1\) gives
\(a_1=1\) and \(a_2=0\). The linear register
\[
y_0'=y_2,\qquad y_1'=y_0+y_2,\qquad y_2'=y_1
\]
has the \(7=2^3-1\) cycle
\[
100\to010\to001\to110\to011\to111\to101\to100.
\]
The Pulse implementation replaces the tapped XOR stage by XNOR:
\[
x_0'=x_2,\qquad x_1'=1+x_0+x_2,\qquad x_2'=x_1,
\]
which is an affine translate of the same maximum-period register.
\end{example}

Figure~\ref{fig:affine-feedback-register} gives a graph-level schematic of the affine feedback mechanism used below.

\begin{figure}[!htbp]
\centering
\begin{tikzpicture}[
>=stealth,
vertex/.style={
circle,
draw,
minimum size=8mm,
inner sep=0pt
},
tap/.style={
circle,
draw,
double,
double distance=1pt,
minimum size=8mm,
inner sep=0pt
},
control/.style={
circle,
draw,
fill=black,
text=white,
minimum size=8mm,
inner sep=0pt
},
register/.style={->,thick},
bias/.style={->,thin},
extra/.style={->,thick,dashed},
return/.style={->,thin,densely dotted},
every node/.style={font=\small}
]

\node[control] (c0) at (-4.3,2.2) {$c_0$};
\node[control] (c1) at (-5.3,3.4) {$c_1$};
\node[control] (c2) at (-3.3,3.4) {$c_2$};

\draw[register] (c0) -- (c1);
\draw[register] (c1) -- (c0);
\draw[register] (c0) -- (c2);
\draw[register] (c2) -- (c0);
\draw[register] (c1) -- (c2);
\draw[register] (c2) -- (c1);

\node[align=center] at (-4.3,4.15)
{active control\\triangle};

\node[vertex] (s0) at (-4.3,0) {$s_0$};
\node[vertex] (s1) at (-2.5,0) {$s_1$};
\node at (-1.3,0) {$\cdots$};
\node[tap] (si) at (0,0) {$s_i$};
\node at (1.3,0) {$\cdots$};
\node[vertex] (sm2) at (2.5,0) {$s_{m-2}$};
\node[vertex] (sm1) at (4.3,0) {$s_{m-1}$};

\draw[register] (s0) -- (s1);
\draw[register] (s1) -- (-1.65,0);
\draw[register] (-0.95,0) -- (si);
\draw[register] (si) -- (0.95,0);
\draw[register] (1.65,0) -- (sm2);
\draw[register] (sm2) -- (sm1);
\draw[register] (sm1) to[bend right=32] (s0);

\draw[bias] (c0) -- (s0);
\draw[bias] (c0) to[bend left=10] (s1);
\draw[bias] (c0) to[bend left=12] (si);
\draw[bias] (c0) to[bend left=14] (sm2);
\draw[bias] (c0) to[bend left=16] (sm1);

\draw[extra] (c1) to[bend right=13] (si);
\draw[extra] (c2) to[bend left=13] (si);
\draw[extra] (sm1) to[bend left=28] (si);

\node[align=right, font=\scriptsize] at (4.5,2.95)
{Only one tapped stage is drawn;\\
dashed edges to other tapped stages are omitted.};

\draw[register] (-4.2,-1.35) -- (-3.3,-1.35);
\node[anchor=west] at (-3.15,-1.35) {register edge};

\draw[bias] (-0.6,-1.35) -- (0.3,-1.35);
\node[anchor=west] at (0.45,-1.35) {\(c_0\) bias to every signal vertex};

\draw[extra] (-4.2,-1.85) -- (-3.3,-1.85);
\node[anchor=west] at (-3.15,-1.85) {extra inputs go to stages \(s_i\) with \(a_i=1\), where \(1\le i\le m-2\)};

\end{tikzpicture}
\caption{Schematic realization of the affine feedback register. Every signal stage receives its predecessor and the active control \(c_0\). If \(a_i=1\), the stage also receives \(c_1,c_2\), and \(s_{m-1}\), so it computes XNOR. Only one tapped stage is drawn; analogous dashed edges to other tapped stages are omitted in this example figure. Cohen's trace-zero condition gives \(a_{m-1}=0\), so the final stage \(s_{m-1}\) is not tapped and no self-loop is needed.}
\label{fig:affine-feedback-register}
\end{figure}
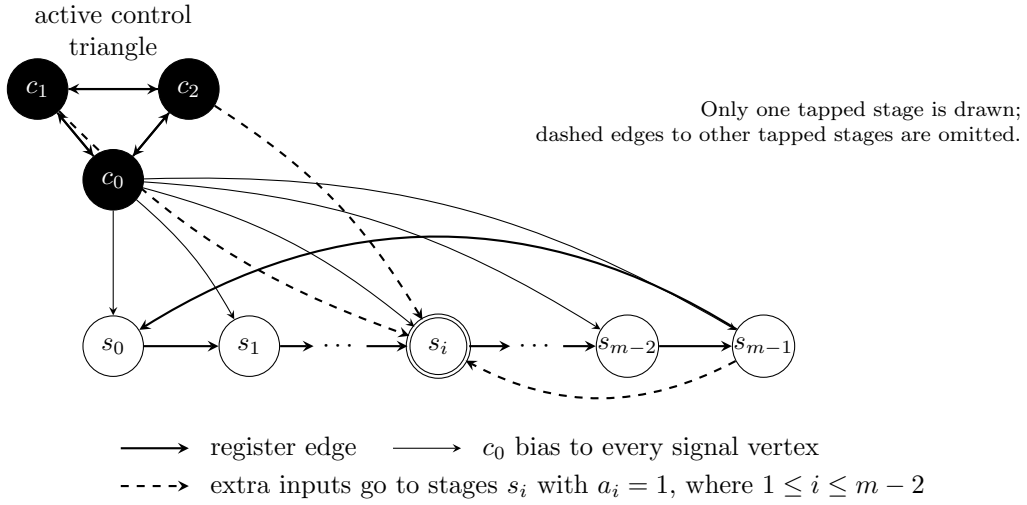

This construction uses the same control triangle and \(c_0\)-bias mechanism as the rotating register; the only new feature is that tapped stages receive the extra inputs \(c_1,c_2\), and \(s_{m-1}\), causing them to compute XNOR.

We now explicitly realize an affine version of \(A\) as a Pulse Graph. Take three control vertices \(c_0,c_1,c_2\) and signal vertices \(s_0,\ldots,s_{m-1}\). Join every pair of distinct controls in both directions and initialize all controls active. Each control then receives exactly two active inputs, so all controls remain active because \(2\) is
prime.

Give \(s_0\) inputs \(c_0\) and \(s_{m-1}\). Then
\[
x_{t+1}(s_0)=x_t(s_{m-1}).
\]
For \(1\le i<m\), if \(a_i=0\), give \(s_i\) inputs \(c_0\) and \(s_{i-1}\), so
\[
x_{t+1}(s_i)=x_t(s_{i-1}).
\]
If \(a_i=1\), give \(s_i\) inputs \(c_0,c_1,c_2,s_{i-1},s_{m-1}\). Its active-input count is
\[
3+x_t(s_{i-1})+x_t(s_{m-1}),
\]
so, since \(3,5\) are prime and \(4\) is not,
\[
x_{t+1}(s_i)=1+x_t(s_{i-1})+x_t(s_{m-1})
\qquad \text{in }\mathbb F_2.
\]
Thus tapped stages compute XNOR. Since \(a_{m-1}=0\), the final stage is not tapped, so no self-loop is needed.

Writing
\[
x_t=
\bigl(
x_t(s_0),\ldots,x_t(s_{m-1})
\bigr)^{\mathsf{T}},
\]
the signal update is the affine transformation
\[
x_{t+1}=T(x_t)=Ax_t+b,
\]
where
\[
b=(0,a_1,\ldots,a_{m-1})^{\mathsf{T}}.
\]

Since \(p\) is primitive and \(m\geq 3\), the number \(1\) is not a root of \(p\). Hence
\[
\det(I-A)=p(1)\neq 0,
\]
so \(I-A\) is invertible. There is a unique vector \(z\in\mathbb{F}_2^m\) such that
\[
(I-A)z=b.
\]

Set
\[
y_t=x_t+z.
\]
Since \(Az+b=z\), we obtain
\[
\begin{aligned}
y_{t+1}
&=x_{t+1}+z\\
&=A(y_t+z)+b+z\\
&=Ay_t.
\end{aligned}
\]
Thus \(T\) is conjugate by translation to \(A\).

The affine map \(T\) therefore has one fixed point and one cycle containing all its other
\[
2^m-1
\]
states. With the controls fixed, the Pulse Graph has an attractor of exact
period
\[
2^m-1=2^{n-3}-1.
\]
This proves the lower bound.

Every signal vertex has in-degree either two or five, while every control vertex has in-degree two. If
\[
r=\bigl|\{i:1\leq i<m,\ a_i=1\}\bigr|,
\]
then \(r\leq m-2\), since \(a_{m-1}=0\). The graph has six edges in the control triangle, \(m\) copy-bias edges from \(c_0\), \(m\) register edges, and \(3r\) additional XNOR inputs. Thus its number of edges is
\[
\begin{aligned}
6+m+m+3r
&\leq 6+2m+3(m-2)\\
&=5m\\
&=5n-15.
\end{aligned}
\]

Finally, for \(n\geq 5\),
\[
2^{n-3}-1\geq 2^{n-4}=\frac{1}{16}2^n.
\]
Together with the upper bound, this proves
\[
L(n)=\Theta(2^n).
\]
\end{proof}

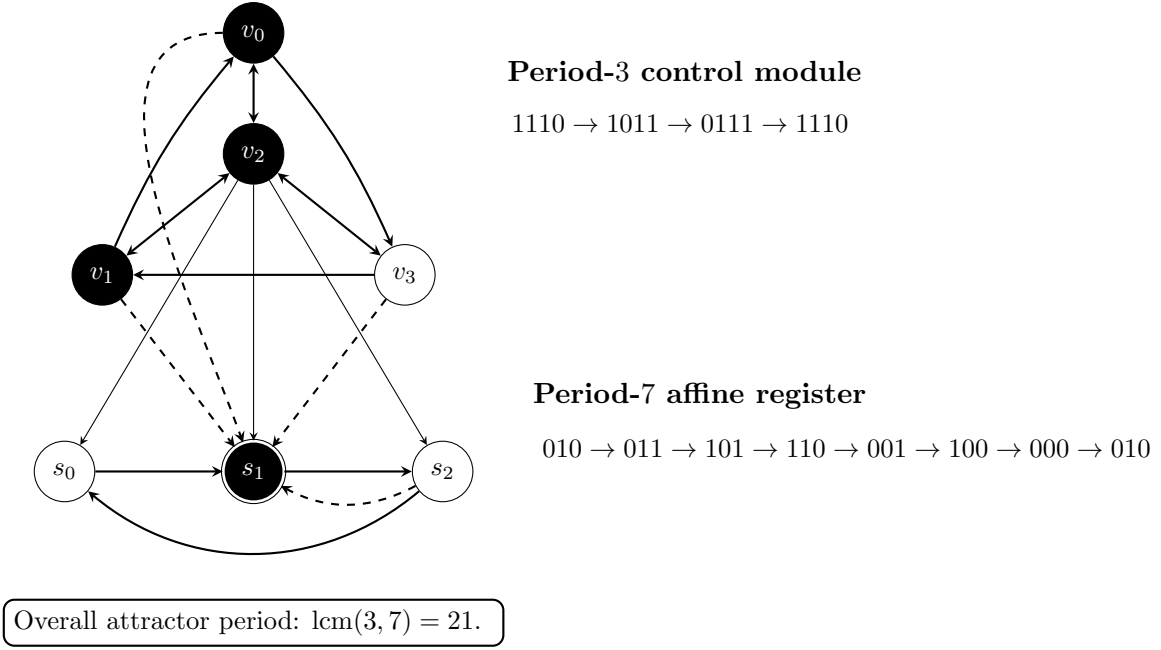
\begin{figure}[!htbp]
\centering
\begin{tikzpicture}[
>=stealth,
vertex/.style={
    circle,
    draw,
    minimum size=8mm,
    inner sep=0pt
},
tap/.style={
    circle,
    draw,
    double,
    fill=black,
    text=white,
    double distance=1pt,
    minimum size=8mm,
    inner sep=0pt
},
active/.style={
    circle,
    draw,
    fill=black,
    text=white,
    minimum size=8mm,
    inner sep=0pt
},
edge/.style={->,thick},
bias/.style={->,thin},
extra/.style={->,thick,dashed},
every node/.style={font=\small}
]

\node[active]       (v0) at (0,4.2) {$v_0$};
\node[active]       (v2) at (0,2.6) {$v_2$};
\node[active]       (v1) at (-2.0,1.0) {$v_1$};
\node[vertex]       (v3) at (2.0,1.0) {$v_3$};

\draw[<->,thick] (v2) -- (v0);
\draw[<->,thick] (v2) -- (v1);
\draw[<->,thick] (v2) -- (v3);

\draw[edge] (v1) to[bend left=8] (v0);
\draw[edge] (v0) to[bend left=8] (v3);
\draw[edge] (v3) -- (v1);

\node[font=\bfseries] at (5.7,3.7) {Period-\(3\) control module};

\node[align=center] at (5.7,3.0)
{
$1110\to1011\to0111\to1110$
};

\node[vertex] (s0) at (-2.5,-1.6) {$s_0$};
\node[tap]    (s1) at (0,-1.6) {$s_1$};
\node[vertex] (s2) at (2.5,-1.6) {$s_2$};

\draw[edge] (s0) -- (s1);
\draw[edge] (s1) -- (s2);
\draw[edge] (s2) to[bend left=40] (s0);

\draw[extra] (s2) to[bend left=28] (s1);

\node[font=\bfseries] at (5.9,-0.6) {Period-\(7\) affine register};

\node[align=center] at (7.9,-1.3)
{
$010\to011\to101\to110\to001\to100\to000\to010$
};

\draw[bias] (v2) -- (s0);
\draw[bias] (v2) -- (s1);
\draw[bias] (v2) -- (s2);

\draw[extra] (v0) to[out=-180,in=110] (s1);
\draw[extra] (v1) -- (s1);
\draw[extra] (v3) -- (s1);

\node[draw,rounded corners,thick,align=center] at (0,-3.6)
{
Overall attractor period:
$\operatorname{lcm}(3,7)=21$.
};

\end{tikzpicture}
\caption{A \(7\)-vertex Pulse Graph obtained by stacking the period-\(3\) control module above a period-\(7\) affine register. The vertex \(v_2\) points to all signal vertices, while \(v_0,v_1,v_3\) feed into the tapped stage \(s_1\). Together with the feedback input from \(s_2\), this makes \(s_1\) compute XNOR, yielding a period-\(7\) signal cycle. Since the control module has period \(3\), the combined system has period \(21\).}
\label{fig:n7-period21-vertical}
\end{figure}

\begin{corollary}
For every odd \(n\ge 7\),
\[
L(n)\ge 3\left(2^{n-4}-1\right).
\]
\end{corollary}

\begin{proof}
Set \(m=n-4\). Initialize the control module on its period-\(3\) orbit and the affine register on its period-\(2^m-1\) orbit. Along the period-\(3\) control orbit, \(v_2\) is always active and exactly three control vertices are active. Hence untapped signal stages still compute copy, while tapped stages still compute XNOR, so the signal register has period \(2^m-1\). Since \(m\) is odd,
\[
2^m-1\equiv1\pmod3.
\]
Thus the control and signal periods are coprime, and the combined period is \(3(2^m-1)\).
\end{proof}

\section{Complete Directed Pulse Graphs}
\label{sec:complete-graphs}

Let \(K_n^\rightarrow\) denote the complete directed graph on \(n\) vertices with no self-loops: for every pair \(u\neq v\), the edge \(u\to v\) is present.

\begin{theorem}[Complete graph update formula]
Let \(S_t\subseteq V\) be the active set at time \(t\), and let \(m=|S_t|\). In \(K_n^\rightarrow\), the next active set is
\[
S_{t+1}
=
\begin{cases}
V, & \text{if } m-1 \text{ is prime and } m \text{ is prime},\\
S_t, & \text{if } m-1 \text{ is prime and } m \text{ is not prime},\\
V\setminus S_t, & \text{if } m-1 \text{ is not prime and } m \text{ is prime},\\
\varnothing, & \text{if } m-1 \text{ is not prime and } m \text{ is not prime}.
\end{cases}
\]
\end{theorem}

\begin{proof}
If \(v\in S_t\), then \(v\) has exactly \(m-1\) active in-neighbors, since self-loops are excluded. Hence active vertices remain active exactly when \(m-1\) is prime.

If \(v\notin S_t\), then \(v\) has exactly \(m\) active in-neighbors. Hence inactive vertices become active exactly when \(m\) is prime.

Combining these two facts gives the four cases.
\end{proof}

\begin{theorem}[Attractor classification for complete Pulse Graphs]
In \(K_n^\rightarrow\), every orbit eventually enters either a fixed point or a complement two-cycle \(S \longleftrightarrow V\setminus S\).
\end{theorem}

\begin{proof}
By the complete graph update formula, every active set \(S\subseteq V\) maps to exactly one of
\[
\varnothing,\qquad S,\qquad V\setminus S,\qquad V.
\]
The fixed points are exactly:
\begin{enumerate}
    \item the empty state \(\varnothing\), which is always fixed;
    \item the all-ones state \(V\), which is fixed exactly when \(n-1\) is prime;
    \item proper nonempty sets \(S\subsetneq V\) satisfying
    \[
    |S|-1 \text{ is prime}
    \qquad\text{and}\qquad
    |S| \text{ is not prime}.
    \]
\end{enumerate}

The complement two-cycles are exactly those
\[
S \longleftrightarrow V\setminus S
\]
where
\[
|S| \text{ and } |V\setminus S| \text{ are prime,}
\]
and neither size equals \(3\).

Indeed, \(S\mapsto V\setminus S\) occurs exactly when inactive vertices turn on and active vertices turn off. In \(K_n^\rightarrow\), this means
\[
|S| \text{ is prime}
\qquad\text{and}\qquad
|S|-1 \text{ is not prime}.
\]
For prime \(|S|\), the second condition fails only when \(|S|=3\). Applying the same condition to \(V\setminus S\) gives the stated condition for a complement two-cycle.
\end{proof}

\begin{definition}
Fix \(K_n^\rightarrow\) and an initial active set \(S_0\subseteq V\). Starting from \(S_0\), the Pulse update produces a sequence
\[
S_0,S_1,S_2,\dots.
\]
Since there are only finitely many possible active sets, this sequence eventually repeats.

The \emph{transient time} \(t(S_0)\) is the number of steps before the sequence reaches its eventual repeating part. Equivalently, \(t(S_0)\) is the first time \(t\) such that \(S_t\) belongs to the final periodic orbit.

The \emph{maximum transient time} is
\[
T(n)=\max_{S_0\subseteq V} t(S_0).
\]
Thus \(T(n)\) is the worst-case number of steps needed, over all initial active sets, before the dynamics reach their eventual fixed point or cycle.
\end{definition}

\begin{theorem}[Constant transient bound for complete Pulse Graphs]

For the complete directed loopless Pulse Graph \(K_n^\rightarrow\),
\[
T(n)\leq 3.
\]
Thus the maximum transient time is bounded independently of \(n\). In particular,
\[
T(n)=O(1).
\]
\end{theorem}

\begin{proof}
Let \(S_0\subseteq V\) be arbitrary. By the complete graph update formula,
\[
S_1\in\{\varnothing,S_0,V\setminus S_0,V\}.
\]

We consider the four possibilities.

\emph{Case 1: \(S_1=\varnothing\).}
The empty state is fixed, since every vertex sees \(0\) active in-neighbors and \(0\) is not prime. Hence the orbit has entered an attractor at time \(1\).

\emph{Case 2: \(S_1=S_0\).}
Then \(S_0\) is a fixed point. Hence the orbit was already on an attractor at time \(0\).

\emph{Case 3: \(S_1=V\).}
If \(n-1\) is prime, then \(V\) is fixed, so the orbit enters an attractor at time \(1\). If \(n-1\) is not prime, then \(V\mapsto\varnothing\), and \(\varnothing\) is fixed. Hence the orbit enters an attractor by time \(2\).

\emph{Case 4: \(S_1=V\setminus S_0\).}
Apply the complete graph update formula again:
\[
S_2\in\{\varnothing,V\setminus S_0,S_0,V\}.
\]

If \(S_2=\varnothing\), then the orbit enters the fixed point \(\varnothing\) at time \(2\).

If \(S_2=V\setminus S_0\), then \(V\setminus S_0\) is fixed, so the orbit enters an attractor at time \(1\).

If \(S_2=S_0\), then
\[
S_0\longmapsto V\setminus S_0\longmapsto S_0,
\]
so the orbit is on a two-cycle from time \(0\).

If \(S_2=V\), then either \(V\) is fixed or \(V\mapsto\varnothing\). Therefore the orbit enters an attractor by time \(3\).

Thus every orbit enters a fixed point or a two-cycle after at most three updates.
\end{proof}

\subsection{Counting Attractors in Complete Pulse Graphs}

The classification theorem also gives explicit formulas for the number of attractors in \(K_n^\rightarrow\).

Let
\[
F(n)=\text{number of fixed points of }K_n^\rightarrow,
\]
\[
C(n)=\text{number of complement two-cycles of }K_n^\rightarrow,
\]
and
\[
A(n)=\text{total number of attractors of }K_n^\rightarrow.
\]

\begin{proposition}[Fixed-point count]
For the complete directed loopless Pulse Graph \(K_n^\rightarrow\),
\[
F(n)
=
1
+
\begin{cases}
1, & \text{if } n-1 \text{ is prime},\\
0, & \text{otherwise},
\end{cases}
+
\sum_{\substack{1\leq m\leq n-1\\
m-1\text{ is prime}\\
m\text{ is not prime}}}
\binom{n}{m}.
\]

\end{proposition}

\begin{proof}
The empty state \(\varnothing\) is always fixed, giving the first term \(1\).

The full state \(V\) contributes
\[
\begin{cases}
1, & \text{if } n-1 \text{ is prime},\\
0, & \text{otherwise}.
\end{cases}
\]

For a proper nonempty active set \(S\), let \(m=|S|\). Such a set is fixed exactly when
\[
m-1 \text{ is prime}
\qquad\text{and}\qquad
m \text{ is not prime}.
\]
For each allowed size \(m\), there are \(\binom{n}{m}\) choices of \(S\). Summing over all such \(m\) gives the formula.
\end{proof}

\begin{proposition}[Two-cycle count]
The number of complement two-cycles in \(K_n^\rightarrow\) is
\[
C(n)
=
\frac{1}{2}
\sum_{\substack{0\leq m\leq n\\
m\text{ prime}\\
n-m\text{ prime}\\
m\neq 3,\ n-m\neq 3}}
\binom{n}{m}.
\]

\end{proposition}

\begin{proof}
A complement two-cycle has the form
\[
S\longleftrightarrow V\setminus S.
\]
By the attractor classification theorem, this occurs exactly when \(|S|\) and \(|V\setminus S|\) are both prime, and neither size equals \(3\).

If \(|S|=m\), then \(|V\setminus S|=n-m\). There are \(\binom{n}{m}\) choices for \(S\). The restrictions in the sum enforce the conditions that \(m\) and \(n-m\) are prime and neither equals \(3\).

Finally, each two-cycle is counted twice: once by \(S\), and once by \(V\setminus S\). Therefore we divide by \(2\).
\end{proof}

\begin{corollary}[Total attractor count]
The total number of attractors of \(K_n^\rightarrow\) is
\[
A(n)=F(n)+C(n).
\]
\end{corollary}

The first 10 values are:
\[
\begin{array}{c|cccccccccc}
n & 1&2&3&4&5&6&7&8&9&10\\
\hline
F(n) & 1&1&2&2&6&17&43&100&220&466\\
C(n) & 0&0&0&3&0&0&21&0&36&126\\
A(n) & 1&1&2&5&6&17&64&100&256&592
\end{array}
\]

\section{Local Activation Probability and Mean-Field Heuristic}

\subsection{Local Activation Probability}
For a vertex of indegree \(d\), suppose each incoming neighbor is independently active with probability \(\rho\). Then the number of active in-neighbors is binomial:
\[
X\sim \operatorname{Binomial}(d,\rho).
\]
Hence the probability that the vertex turns on is
\[
P_d(\rho)
=
\mathbb{P}(X\text{ is prime})
=
\sum_{\substack{0\leq k\leq d\\ k\text{ prime}}}
\binom{d}{k}\rho^k(1-\rho)^{d-k}.
\]
In the uniform case \(\rho=1/2\),
\[
P_d(1/2)
=
\frac{1}{2^d}
\sum_{\substack{0\leq k\leq d\\ k\text{ prime}}}
\binom{d}{k}.
\]
For example,
\[
P_1(1/2)=0,\qquad
P_2(1/2)=\frac14,\qquad
P_3(1/2)=\frac12,\qquad
P_4(1/2)=\frac58.
\]

\subsection{Sparse Random Pulse Graph Mean-field Heuristic}

Let \(G\) be a random directed graph on \(n\) vertices with no self-loops. Each possible directed edge appears independently with probability \(c/n\). Since there are \(n-1\) possible incoming edges to each vertex, the expected indegree of one vertex is
\[
(n-1)\frac cn \approx c.
\]
Thus \(c\) represents the average indegree.

Let \( \rho_t \) be the fraction of active vertices at time \(t\). For example, \(\rho_t=0.25\) means that about \(25\%\) of the vertices are active.

Fix a typical vertex \(v\). A different vertex \(u\) gives an active input to \(v\) when \(u\) is active and the edge \(u\to v\) exists. This has probability approximately 
\[
\rho_t\frac cn.
\]
Since there are about \(n-1\) possible choices of \(u\), the active in-neighbor count is approximately
\[
\operatorname{Binomial}\left(n-1,\frac{c\rho_t}{n}\right).
\]
For large \(n\), this is approximately 
\[
\operatorname{Poisson}(c\rho_t).
\]

Define \(
\lambda=c\rho_t.
\) Here \(\lambda\) is the expected number of active in-neighbors of a typical vertex.

Since a vertex turns on exactly when its number of active in-neighbors is prime, the next active density is approximately
\[
\rho_{t+1}=Q(c\rho_t),
\]
where
\[
Q(\lambda)
=
\mathbb{P}(\operatorname{Poisson}(\lambda)\text{ is prime})
=
e^{-\lambda}\sum_{p\ \mathrm{prime}}\frac{\lambda^p}{p!}.
\]

Near \(\rho=0\), the first prime contribution is \(p=2\), so
\[
Q(c\rho)
\sim
\frac{(c\rho)^2}{2}.
\]
Therefore very small activity tends to shrink rather than grow. In the mean-field model, the all-zero state is locally stable for every finite \(c\).

A nonzero fixed active density satisfies
\[
\rho=Q(c\rho).
\]
Writing
\[
\lambda=c\rho,
\]
we get
\[
\rho=Q(\lambda),
\qquad
c=\frac{\lambda}{Q(\lambda)}.
\]

\begin{theorem}[Mean-field threshold for positive fixed points]
\label{thm:mean-field-threshold}

Define
\[
\Phi(\lambda)=\sum_{p\ \mathrm{prime}}\frac{\lambda^p}{p!},
\qquad
Q(\lambda)=e^{-\lambda}\Phi(\lambda),
\qquad
c(\lambda)=\frac{\lambda}{Q(\lambda)}
\quad (\lambda>0).
\]
Then \(c(\lambda)\) has a unique global minimizer
\[
\lambda_\ast\approx 1.408509.
\]
At this point,
\[
\rho_\ast=Q(\lambda_\ast)
\approx 0.368241,
\qquad
c_\ast=c(\lambda_\ast)
\approx 3.824963.
\]
Equivalently, \(c_\ast\) is the global threshold for the existence of a positive solution of the mean-field fixed-point equation
\[
\rho=Q(c\rho).
\]
The minimizer is unique, although \(c(\lambda)\) may have additional
critical points for larger values of \(\lambda\).
\end{theorem}

\begin{proof}[Proof sketch]
By the preceding parameterization, the threshold is the global minimum of \(c(\lambda)\). Differentiation shows that its critical points satisfy
\[
\Omega(\lambda)=
\lambda\Phi'(\lambda)-(1+\lambda)\Phi(\lambda)=0.
\]
One proves that \(\Omega\) has a unique zero in \((1,2)\), that this zero gives the global minimum, and that no \(\lambda\geq5\) gives a smaller value because \(c(\lambda)>\lambda\). The numerical values follow from solving the critical-point equation. Full details are given in Appendix~\ref{app:mean-field}.
\end{proof}

\begin{remark}
The theorem proves the global minimum of \(c(\lambda)\), not that \(\lambda_\ast\) is the only critical point on \((0,\infty)\). Additional critical points at larger \(\lambda\) do not affect the global threshold for positive fixed points of the mean-field map, because \(c(\lambda)\geq\lambda\).
\end{remark}

Since
\[
c(\lambda)\to\infty
\quad\text{as }\lambda\downarrow0
\qquad\text{and}\qquad
c(\lambda)\to\infty
\quad\text{as }\lambda\to\infty,
\]
continuity and the uniqueness of the global minimum imply that the mean-field fixed-point equation has no positive solution when \(c<c_\ast\), exactly one positive solution when \(c=c_\ast\), and at least two positive solutions when \(c>c_\ast\). The stability of these positive fixed points is not determined by the existence argument alone.

\begin{proposition}[Local bistability above the threshold]
\label{prop:mean-field-bistability}

Let
\[
f_c(\rho)=Q(c\rho).
\]
There exists \(\varepsilon>0\) such that, for every
\[
c_\ast<c<c_\ast+\varepsilon,
\]
the map \(f_c\) has two positive fixed points near \(\rho_\ast\). The smaller branch is unstable, while the larger branch is locally asymptotically stable. Since \(\rho=0\) is also locally asymptotically stable, the mean-field map is bistable in this parameter range.
\end{proposition}

\begin{proof}[Proof sketch]
At the threshold, the positive fixed point has multiplier \(1\), and the parameterization \(c=c(\lambda)\) has a nondegenerate local minimum. Consequently, two nearby positive branches appear for \(c>c_\ast\). The multiplier can be written as
\[
f_c'(\rho)
=
1+\frac{\Omega(\lambda)}{\Phi(\lambda)}.
\]
Its sign relative to \(1\) shows that the lower branch is unstable and the upper branch is locally stable. The complete argument appears in Appendix~\ref{app:mean-field}.
\end{proof}

\section{Conclusion}

The main result shows that a single arithmetic totalistic rule supports attractors of exponential length on sparse graphs of bounded in-degree. Complete directed graphs, by contrast, collapse to fixed points and complement two-cycles with uniformly bounded transient time.

Several open problems remain. One is to narrow the gap between
\[
2^{n-3}-1 \leq L(n)\leq 2^n-n-1.
\]
Other questions include determining the exact values of \(L(n)\) for \(n\geq6\), describing the full set of periods realizable at fixed \(n\), and identifying graph structures that force either short or long attractors. On the probabilistic side, an important next step is to prove a finite-time limit theorem for the prime--Poisson recursion, and determine whether its predicted bistability persists on long time scales in large sparse random Pulse Graphs.

\section*{Acknowledgments}

I thank Professor Thotsaporn Thanatipanonda for helpful comments on an earlier draft.

\appendix

\section{Details of the Mean-Field Analysis}
\label{app:mean-field}

\subsection{Proof of the mean-field threshold theorem}

\begin{proof}[Proof of Theorem~\ref{thm:mean-field-threshold}]
Since
\[
c(\lambda)=\frac{\lambda e^\lambda}{\Phi(\lambda)},
\]
differentiation gives
\[
\begin{aligned}
c'(\lambda)
&=
\frac{e^\lambda(1+\lambda)\Phi(\lambda)
      -\lambda e^\lambda \Phi'(\lambda)}
     {\Phi(\lambda)^2} \\
&=
-\frac{e^\lambda}{\Phi(\lambda)^2}
\left[
\lambda \Phi'(\lambda)-(1+\lambda)\Phi(\lambda)
\right].
\end{aligned}
\]
Define
\[
\begin{aligned}
\Omega(\lambda)
&=
\lambda \Phi'(\lambda)-(1+\lambda)\Phi(\lambda) \\
&=
\sum_{p\ \mathrm{prime}}
\lambda \, p\frac{\lambda^{p-1}}{p!}
-
\sum_{p\ \mathrm{prime}}
(1+\lambda)\frac{\lambda^p}{p!}\\
&=
\sum_{p\ \mathrm{prime}}
(p-1-\lambda)\frac{\lambda^p}{p!}.
\end{aligned}
\]
Then
\[
c'(\lambda)
=
-\frac{e^\lambda}{\Phi(\lambda)^2}\Omega(\lambda).
\]
Thus, \(c'(\lambda)\) has the opposite sign from \(\Omega(\lambda)\).

For \(0<\lambda\leq1\), every summand of \(\Omega(\lambda)\) is nonnegative, and the \(p=3\) summand is strictly positive. Hence
\[
\Omega(\lambda)>0
\qquad (0<\lambda\leq1),
\]
so \(c'(\lambda) < 0\) and \(c\) is strictly decreasing there.

We next prove that \(\Omega\) is strictly decreasing on \([1,5]\). For an
integer \(k\geq1\), define
\[
g_k(\lambda)
=
\left(k-1-\lambda-\frac{\lambda}{k}\right)
\frac{\lambda^{k-1}}{(k-1)!}.
\]
Then
\[
\Omega'(\lambda)=\sum_{p\ \mathrm{prime}}g_p(\lambda).
\]
For \(1\leq\lambda\leq5\) and \(k\geq7\),
\[
k-1-\lambda-\frac{\lambda}{k}
\geq
k-6-\frac{5}{k}
\geq
\frac{2}{7}>0.
\]
Therefore \(g_k(\lambda)\geq0\) for all \(k\geq7\), and hence
\[
\Omega'(\lambda)
\leq
g_2(\lambda)+g_3(\lambda)+g_5(\lambda)
+\sum_{k=7}^{\infty}g_k(\lambda).
\]
If all positive integers are included instead of only primes, then
\[
\sum_{k=1}^{\infty}g_k(\lambda)=1-e^\lambda.
\]
Consequently,
\[
\Omega'(\lambda)
\leq
1-e^\lambda+2\lambda-g_4(\lambda)-g_6(\lambda).
\]
Using
\[
e^\lambda\geq
\sum_{j=0}^{6}\frac{\lambda^j}{j!},
\]
we obtain
\[
\Omega'(\lambda)
\leq
\frac{\lambda}{120}H(\lambda),
\]
where
\[
H(\lambda)
=
\lambda^5-6\lambda^4+20\lambda^3
-80\lambda^2-60\lambda+120.
\]

Now we show \(H(\lambda)<0\) on \([1,5]\). For \(1\leq\lambda\leq5\), we have
\[ 
\lambda^2-6\lambda+20=
(\lambda-3)^2+11
\leq15.
\]
Therefore
\[
\begin{aligned}
H(\lambda)
&=
\lambda^3(\lambda^2-6\lambda+20)
-80\lambda^2-60\lambda+120\\
&\leq
15\lambda^3-80\lambda^2-60\lambda+120\\
&=
-5+(\lambda-1)(15\lambda^2-65\lambda-125).
\end{aligned}
\]
Since \(0<\lambda\leq5\), we have \(\lambda^2\leq5\lambda\), and hence
\[
15\lambda^2-65\lambda-125
\leq
75\lambda-65\lambda-125=
10\lambda-125
\leq
-75<0.
\]
Also, \(\lambda-1\geq0\). Consequently,
\[
(\lambda-1)(15\lambda^2-65\lambda-125)\leq0,
\]
so
\[
H(\lambda)\leq-5<0
\qquad(1\leq\lambda\leq5).
\]
Therefore
\[
\Omega'(\lambda)<0
\qquad(1\leq\lambda\leq5).
\]

We have \(\Omega(1)>0\). Also,
\[
\Omega(2)
=
-2+
\sum_{\substack{p\geq5\\p\ \mathrm{prime}}}
(p-3)\frac{2^p}{p!}.
\]
All terms in the remaining sum are nonnegative, so enlarging the sum to all integers \(k\geq4\) gives
\[
\Omega(2)
\leq
-2+\sum_{k=4}^{\infty}(k-3)\frac{2^k}{k!}
=
7-e^2<0.
\]
Hence, because \(\Omega\) is strictly decreasing on \([1,5]\), it has a unique zero
\[
\lambda_\ast\in(1,2).
\]
This shows that \(c\) is strictly decreasing on \((0,\lambda_\ast)\) and strictly increasing on \((\lambda_\ast,5]\).

Finally, since \(Q(\lambda)\) is a probability,
\[
0<Q(\lambda)<1
\qquad(\lambda>0).
\]
Therefore
\[
c(\lambda)=\frac{\lambda}{Q(\lambda)}>\lambda.
\]

Thus, \(c(\lambda)>5\) for every \(\lambda\geq5\). On the other hand,
\[
Q(1)=e^{-1}\sum_{p\ \mathrm{prime}}\frac{1^p}{p!}>
e^{-1}\left(\frac{1}{2!}+\frac{1}{3!}\right)
=
\frac{2}{3e},
\]
so
\[
c(\lambda_\ast)< c(1)<\frac{3e}{2}<5.
\]
Therefore, no \(\lambda\geq5\) can satisfy \(c(\lambda)\leq c(\lambda_\ast)\). Hence \(\lambda_\ast\) is the unique global minimizer of \(c(\lambda)\).
\end{proof}

\subsection{Proof of local bistability}

\begin{proof}[Proof of Proposition~\ref{prop:mean-field-bistability}]

First, \(\rho=0\) is locally asymptotically stable. Indeed,
\[
Q(z)=\frac{z^2}{2}+O(z^3)
\qquad (z\to 0),
\]
so
\[
f_c'(0)=cQ'(0)=0.
\]
Hence
\[
\lvert f_c'(0)\rvert<1.
\]

Positive fixed points are parameterized by
\[
\rho(\lambda)=Q(\lambda),
\qquad
c(\lambda)=\frac{\lambda}{Q(\lambda)}.
\]
At \(\lambda=\lambda_\ast\),
\[
c'(\lambda_\ast)=0.
\]
Moreover,
\[
c'(\lambda)
=
-\frac{e^\lambda}{\Phi(\lambda)^2}\Omega(\lambda).
\]
Differentiating and using \(\Omega(\lambda_\ast)=0\) gives
\[
c''(\lambda_\ast)
=
-\frac{e^{\lambda_\ast}}{\Phi(\lambda_\ast)^2}
\Omega'(\lambda_\ast).
\]
Since \(\Omega'(\lambda_\ast)<0\),
\[
c''(\lambda_\ast)>0.
\]
Thus \(c(\lambda)\) has a nondegenerate local minimum at \(\lambda_\ast\). Consequently, for every \(c>c_\ast\) sufficiently close to \(c_\ast\), there are two nearby solutions
\[
\lambda_-(c)<\lambda_\ast<\lambda_+(c)
\]
of
\[
c=c(\lambda).
\]
They produce two positive fixed points
\[
\rho_-(c)=Q(\lambda_-(c)),
\qquad
\rho_+(c)=Q(\lambda_+(c)).
\]

To determine their stability, compute the fixed-point multiplier:
\[
\begin{aligned}
f_c'(\rho)
&=cQ'(\lambda) \\
&=\frac{\lambda Q'(\lambda)}{Q(\lambda)}.
\end{aligned}
\]
Since
\[
Q(\lambda)=e^{-\lambda}\Phi(\lambda),
\]
we have
\[
\frac{Q'(\lambda)}{Q(\lambda)}
=
\frac{\Phi'(\lambda)}{\Phi(\lambda)}-1.
\]
Therefore
\[
\begin{aligned}
f_c'(\rho)
&=
\lambda\frac{\Phi'(\lambda)}{\Phi(\lambda)}-\lambda \\
&=
1+\frac{\Omega(\lambda)}{\Phi(\lambda)}.
\end{aligned}
\]

At \(\lambda=\lambda_\ast\),
\[
f_{c_\ast}'(\rho_\ast)=1.
\]

On the lower branch, \(\lambda_-<\lambda_\ast\), so
\[
\Omega(\lambda_-)>0.
\]
Hence
\[
f_c'(\rho_-)
=
1+\frac{\Omega(\lambda_-)}{\Phi(\lambda_-)}
>1,
\]
and \(\rho_-\) is unstable.

On the upper branch, \(\lambda_+>\lambda_\ast\), so
\[
\Omega(\lambda_+)<0,
\]
and therefore
\[
f_c'(\rho_+)<1.
\]
Because the multiplier is continuous and equals \(1\) at \(\lambda_\ast\), for \(\lambda_+\) sufficiently close to \(\lambda_\ast\) we also have
\[
-1<f_c'(\rho_+)<1.
\]
Thus \(\rho_+\) is locally asymptotically stable.

Therefore, for \(c>c_\ast\) sufficiently close to \(c_\ast\), both \(\rho=0\) and \(\rho=\rho_+(c)>0\) are locally stable, while the intermediate positive branch is unstable. Hence the mean-field map is locally bistable.
\end{proof}

\bibliographystyle{plainnat}
\bibliography{refs}

\end{document}